\documentclass[12pt,a4paper, twoside, reqno]{amsart}
\usepackage{amssymb}
\usepackage{enumitem} 
\usepackage[
top    = 2.75cm,
bottom = 2.50cm,
left   = 3.00cm,
right  = 2.50cm]{geometry}

\usepackage{graphicx}
\usepackage{amsfonts,amsmath,amsthm,amssymb,ifsym,float}
\usepackage{mathrsfs}
\usepackage{tikz}
\usepackage{pdfsync}
\usepackage{latexsym}
\usepackage[bookmarks=true]{hyperref}

\usepackage{setspace}
\usepackage{ifthen}
\usepackage[normalem]{ulem}
\usepackage{bbding}
\usepackage{enumitem}
\usepackage{xcolor}
\renewcommand*{\thefootnote}{\fnsymbol{footnote}}

\newtheorem{theorem}{Theorem}[section]
\newtheorem{theo}[theorem]{Theorem}

\newtheorem{lem}[theorem]{Lemma}
\newtheorem*{lem*}{Lemma}
\newtheorem*{theo*}{Theorem}

\newtheorem{rem}[theorem]{Remark}

\title[]{On solvability of certain equations of arbitrary length over torsion-free groups }
\author{ M. Fazeel Anwar$^{1*}$}
\author{Mairaj Bibi$^{2}$}
\author{ M. Saeed Akram$^{3}$ }

\begin{document}
	\footnotetext[1]{Corresponding Author}
	\renewcommand*{\thefootnote}{\arabic{footnote}}

	\footnotetext[1]{Department of Mathematics, Sukkur IBA University.  e-mail: fazeel.anwar@iba-suk.edu.pk}
	
	\footnotetext[2]{Department of Mathematics, COMSATS Institute of Information Technology, Islamabad. email: mairaj\_maths@comsats.edu.pk }
	
	\footnotetext[3]{Department of Mathematics, Khawaja Fareed UEIT, email: mrsaeedakram@gmail.com}

\begin{abstract}
	Let $G$ be a non-trivial torsion free group and  $s(t)=g_{1}t^{\epsilon_{1}}g_{2}t^{\epsilon_{2}} \cdots g_{n}t^{\epsilon_{n}}=1 \; (g_{i} \in G,\ \epsilon_i=\pm 1)$ be an equation over $G$ containing no blocks of the form $t^{-1}g_{i}t^{-1}, \; g_{i} \in G$. In this paper we show that $s(t)=1$ has a solution over $G$ provided a single relation on coefficients of $s(t)$ holds. We also generalize our results to equations containing higher powers of $t$. The later equations are also related to Kaplansky zero-divisor conjecture \cite{K}.
	
	\medskip
	
	\noindent \textbf{AMSC:} 20F05, 57M05

\end{abstract}

\maketitle

\section{Introduction} 

Let $G$ be a non-trivial group, $t$ be an unknown and let $F$ be a free group generated by $t$.  An equation in $ t$ over $G$ is an expression of the form
$$
s(t)=g_{1}t^{\epsilon_{1}}g_{2}t^{\epsilon_{2}} \cdots g_{n}t^{\epsilon_{n}}=1 \ (g_{i} \in G,\ \epsilon_i=\pm 1)
$$
in which it is assumed that $\epsilon_{i}+\epsilon_{i+1}=0$ implies $g_{i+1} \neq 1$ in $G$. The integer $n$ is called the length of the equation. The equation $s(t)=1$ is said to have a solution over $G$ if there is an embedding $\phi$ of $G$ into a group $H$ and an element $h \in H$ such that $\phi(g_{1})h^{\epsilon_{1}}\phi(g_{2})h^{\epsilon_{2}} \cdots \phi(g_{n})h^{\epsilon_{1}}=1$ in $H$. Equivalently $s(t)=1$ has a solution over $G$ if and only if the natural map from $G$ to $\langle G*F|s(t) \rangle$ is injective, where $G*F$ is the free product of $G$ and $F$. In \cite{levin}, Levin conjectured that every equation is solvable over a torsion free group. A significant work has been done to verify Levin's conjecture. In \cite{P}, Prishchepov used results of Brodskii and Howie \cite{BH} to show that the conjecture is true for $n \leq 6$. A different proof of the same theorem was given by Ivanov and Klyachko \cite{IK}. In \cite{BE}, Bibi and Edjvet proved that the conjecture holds for $n =7$. The authors considered a nonsingular equation of length $8$ in \cite{ABIA} and proved that the conjecture holds. The proofs in \cite{BE} and \cite{ABIA} are given by considering all possible conditions on elements of the group $G$. The number of such cases become extremely large as length of the equation increases. 

In this paper we consider equations of arbitrary length and show that the Levin conjecture holds under some mild conditions. These results can half the number of cases in \cite{ABIA} and can also significantly reduce the number of cases one has to consider for all equations of length greater than or equal to $8$. Our main results are the following

\begin{theo*}
	Let $\displaystyle s(t)=g_{1}E_{k_1}g_{2}E_{k_2} \, \cdots \, g_{n}E_{k_n}$, such that $g_1 ,  \, \cdots  \, , g_n \in G$ and $E_{k_i}=t a_{k_{i-1}+1} t \, \cdots \, t a_{k_i}t^{-1}$ with $k_0 =0$ and $k_i \geq 1$ for all $i \geq 1$. If $g_{j}=g_{1}^{-1}$ for all $j \geq 2$ then $s(t)=1$ has a solution over $G$.
\end{theo*}

\begin{theo*}
	Suppose $m_{1}, m_{2}, \; \cdots \;  ,m_{k_{n}}$ are positive integers such that $m_{k_{1}} > m_{k_i}$ and $m_{k_{i}+1} > m_{k_{1}+1}$ for all $i \geq 2$ with $m_{k_{1}+1}=m_{k_{n}}$. Let $\displaystyle s(t)=g_{1}E_{k_1}g_{2}E_{k_2} \, \cdots \, g_{n}E_{k_n}$ such that $g_1 , g_2  \, \cdots  \,  g_n \in G$ and $E_{k_i}=t^{m_{k_{i-1}+1}} a_{k_{i-1}+1} t^{m_{k_{i-1}+2}} \, \cdots \,  t^{m_{k_{i}-1}} a_{k_i}t^{-m_{k_{i}}}$ with $k_0 =0$ and $k_i \geq 2$ for all $i \geq 1$. If $g_{j}=g_{1}^{-1}$ for all $j \geq 2$ then $s(t)=1$ has a solution over $G$.
\end{theo*}
	
We also take some specific equations of the above type and demonstrate their solvability. In some cases we obtain stronger results for specific equations. 

\section{Preliminaries}
A relative group presentation is a presentation of the form $\mathcal{P}=\langle G, \; x \; | \; r \rangle$ where $r$ is a set of cyclically reduced words in $G * \langle x \rangle$. If the relative presentation is orientable and aspherical then the natural map from $G$ to $\langle G, \; x \; | \; r \rangle$ is injective. In our case $x$ and $r$ consist of the single element $t$ and $s(t)$ respectively, therefore $\mathcal{P}$ is orientable and so asphericity implies $s(t)=1$ is solvable. In this paper we use the weight test to show that $\mathcal{P}$ is aspherical \cite{BP}.

The star graph $\Gamma$ of $\mathcal{P}$ has vertex set $x \cup x^{-1}$ and edge set $r^{*}$, where $r^{*}$ is the set of all cyclic permutations of the elements of $r \cup r^{-1}$ which begin with an element of $x \cup x^{-1}$. For $R \in r^{*}$ write $R=Sg$ where $g \in G$ and $S$ begins and ends with $x$ symbols. Then $\mathfrak{i}(R)$ is the inverse of the last symbol of $S$, $\tau(R)$ the first symbol of $S$ and $\lambda(R)=g$. A weight function $\theta$ on $\Gamma$ is a real valued function on the set of edges of $\Gamma$ which satisfies $\theta(Sh)=\theta(S^{-1}h^{-1})$. A weight function $\theta$ is called aspherical if the following three conditions are satisfied

\begin{enumerate}
	\item Let $R \in r^{*}$ with $R=x_{1}^{\epsilon_{1}}g_{1} \; \cdots \; x_{n}^{\epsilon_{n}}g_{n}$. Then 
	$$\sum_{i=1}^{n}\; (1-\theta(x_{i}^{\epsilon_{i}}g_{i} \; \cdots \; x_{n}^{\epsilon_{n}}g_{n}x_{1}^{\epsilon_{1}} g_{1} \; \cdots \; x_{i-1}^{\epsilon_{i-1}} g_{i-1})) \geq 2. $$
	
	\item Each admissible cycle in $\Gamma$ has weight at least $2$ (where admissible means having a label trivial in $G$).
	
	\item Each edge of $\Gamma$ has a non-negative weight.
\end{enumerate}    

If $\Gamma$ admits an aspherical weight function then $\mathcal{P}$ is aspherical \cite{BP}. The following lemma \cite{K} tells us that we can apply asphericity test in $k-$steps.

\begin{lem*}
Let the relative presentation $P = \langle H, x : r \rangle$ define a group $G$
and let $Q = \langle G, t : s \rangle$ be another relative presentation. If $Q$ and $P$ are both aspherical,
then the relative presentation $R = \langle H, x \cup t : r \cup \tilde{s} \rangle$ is aspherical, where $\tilde{s}$ is an element
of $H * F(x) * F(t)$ obtained from $s$ by lifting.
\end{lem*}

It is clear from our definition of a group equation that if $g_{i}$ is a coefficient between a negative and a positive power of $t$ than $g_{i}$ is not trivial in $G$. This fact will be used in all subsequent proofs without reference. 


\section{Main Results}

We start by solving a general equation containing two negative powers and then give a similar result for three negative powers. We also give the corresponding results for higher powers of $t$. 

\begin{lem}\label{lsgtype1} The equation $g_1tg_2t \cdots g_{i-1}t^{-1}g_i t \cdots g_n t^{-1}=1$ is solvable if $g_i =g_1^{-1}$.
\end{lem}
\begin{proof}

	Let $g_i =g_1^{-1}$. We substitute $x=t^{-1} g_1^{-1} t$ to get $$\mathcal{P}=\langle A, t~|~ x^{-1}g_2t \cdots g_{i-1} x g_{i+1} t \cdots t g_n =1=t^{-1}g_{1}^{-1}tx^{-1}  \rangle.$$  We use the weight test to show that $\mathcal{P}$ is aspherical.  The star graph $\Gamma$ for $\mathcal{P}$ is given by Figure  \ref{sgtype1}. We assign a weight function $\theta$ such that $\theta(g_1^{-1})=\theta(g_{2})=\theta(g_{n})=0$. The weight of the edge $t^{-1}   \leftrightarrow   x^{-1}$ with label $1$ is also zero. All other edges are assigned a weight $1$. Then $\Sigma (1-\theta(\alpha_{i}))=\Sigma(1-\theta(\beta_{j}))=2$, where $\alpha_{i}$ represents one of the solid edges and $\beta_{j}$ is a dotted edge. Hence (W1) is satisfied.  Now it is clear from the star graph that all admissible cycles of weight less than two has label $g_{n}^{m}=1$ or $g_{1}^{m}=1, \, (m \neq 0)$. Since $G$ is torsion free (W2) is satisfied. Moreover (W3) clearly holds.  Hence $\mathcal{P}$ is aspherical over a torsion free group.

\begin{figure}
	\begin{center}
			\vspace{-1.0in}
		\hspace*{-1.5in}
		\includegraphics[width=0.75\textwidth, height=0.75\textheight]{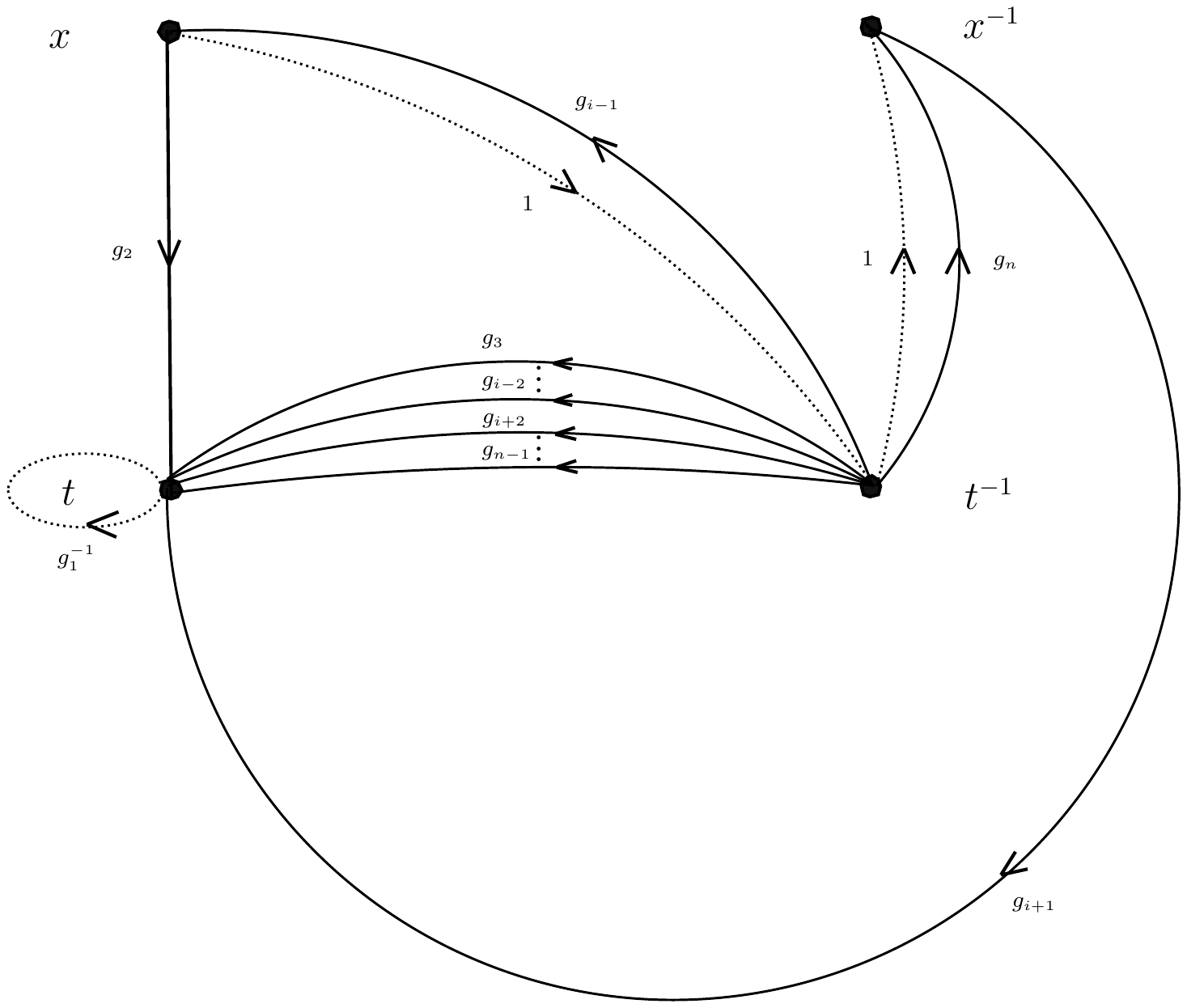}
		\vspace{-3.1in}
		\caption{Star graph $\Gamma$}
		\label{sgtype1}
	\end{center}
\end{figure}

\end{proof}

\begin{lem}\label{lsgtype4} Let $m_1, \cdots , m_n $ be positive integers such that $m_1 \geq m_{i-1}$. Then the equation $$g_1t^{m_1}g_2t^{m_2} \cdots g_{i-1}t^{-m_{i-1}}g_i t^{m_i} \cdots g_n t^{-m_n}=1$$ is solvable if $g_i =g_1^{-1}$.
\end{lem}
\begin{proof}
	
	Let $$\mathcal{P}=\langle A, t~|~ g_1t^{m_1}g_2t^{m_2} \cdots g_{i-1}t^{-m_{i-1}}g_1^{-1} t^{m_i} \cdots g_n t^{-m_n}  \rangle$$ be the relative presentation corresponding to the given equation. Following  \cite{BP}, it is sufficient to show that the presentation $\mathcal{P}$ is aspherical. Substitute $x=t^{-m_{i-1}} g_1^{-1} t^{m_i}$ to get $$\mathcal{P}=\langle A, t~|~ t^{m_i - m_n} x^{-1} t^{m_1 - m_{i-1}}g_2t^{m_2} \cdots g_{i-1} x g_{i+1} t^{m_{i+1}} \cdots t^{m_{n-1}} g_n=1  =t^{-m_{i-1}}g_{1}^{-1}t^{m_{i}}x^{-1}  \rangle.$$ We use the weight test to show that $\mathcal{P}$ is aspherical. The proof is given in separate cases.
	
	\begin{enumerate}
		\item Let $m_i > m_n $ and $m_1 > m_{i-1}$. The star graph $\Gamma$ for $\mathcal{P}$ is given by Figure \ref{sgtype4_1_3}.
		
		\begin{figure}
			\begin{center}
				\vspace{-1.0in}
				\hspace*{-1.5in}
				\includegraphics[width=0.75\textwidth, height=0.75\textheight]{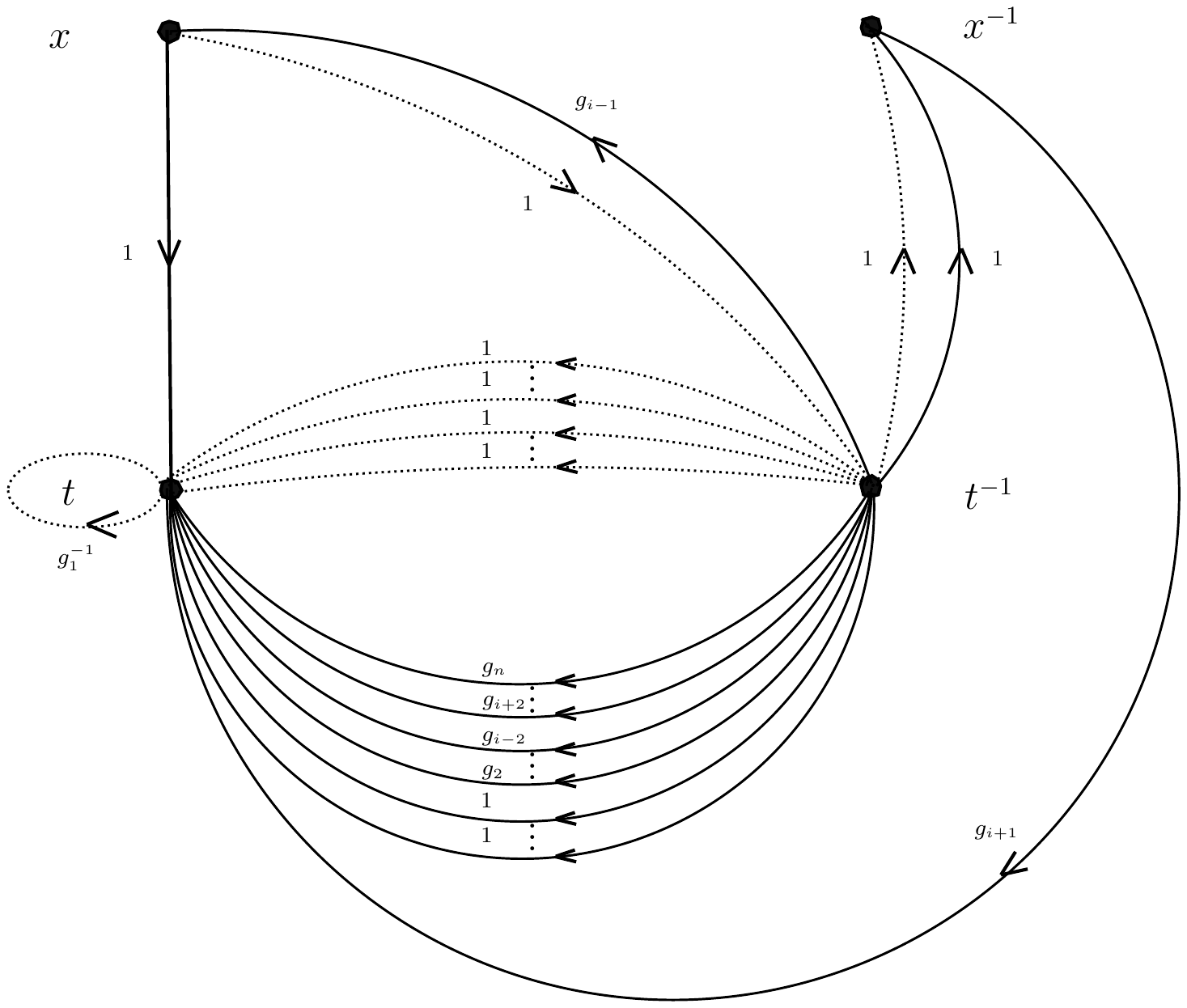}
				\vspace{-3.1in}
				\caption{Star graph $\Gamma$}
				\label{sgtype4_1_3}
				
			\end{center}
		\end{figure} 
		
%

	 We assign a weight function $\theta$ such that $\theta(g_1^{-1})=\theta(g_{i+1})=\theta(g_{i-1})=0$. Moreover the weight of the edge $x   \leftrightarrow   t^{-1}$ with label $1$ is also zero. All other edges are assigned a weight $1$. It is clear from the star graph that all admissible cycles of weight less than two imply that the group is a torsion group. Hence $\mathcal{P}$ is aspherical over a torsion free group.
	 
	 \item Let $m_i < m_n $ and $m_1 > m_{i-1}$. The star graph $\Gamma$ for $\mathcal{P}$ is given by Figure \ref{sgtype4_1_2}.
	 
	 	\begin{figure}
	 	\begin{center}
	 		\vspace{-1.0in}
	 		\hspace*{-1.5in}
	 		\includegraphics[width=0.75\textwidth, height=0.75\textheight]{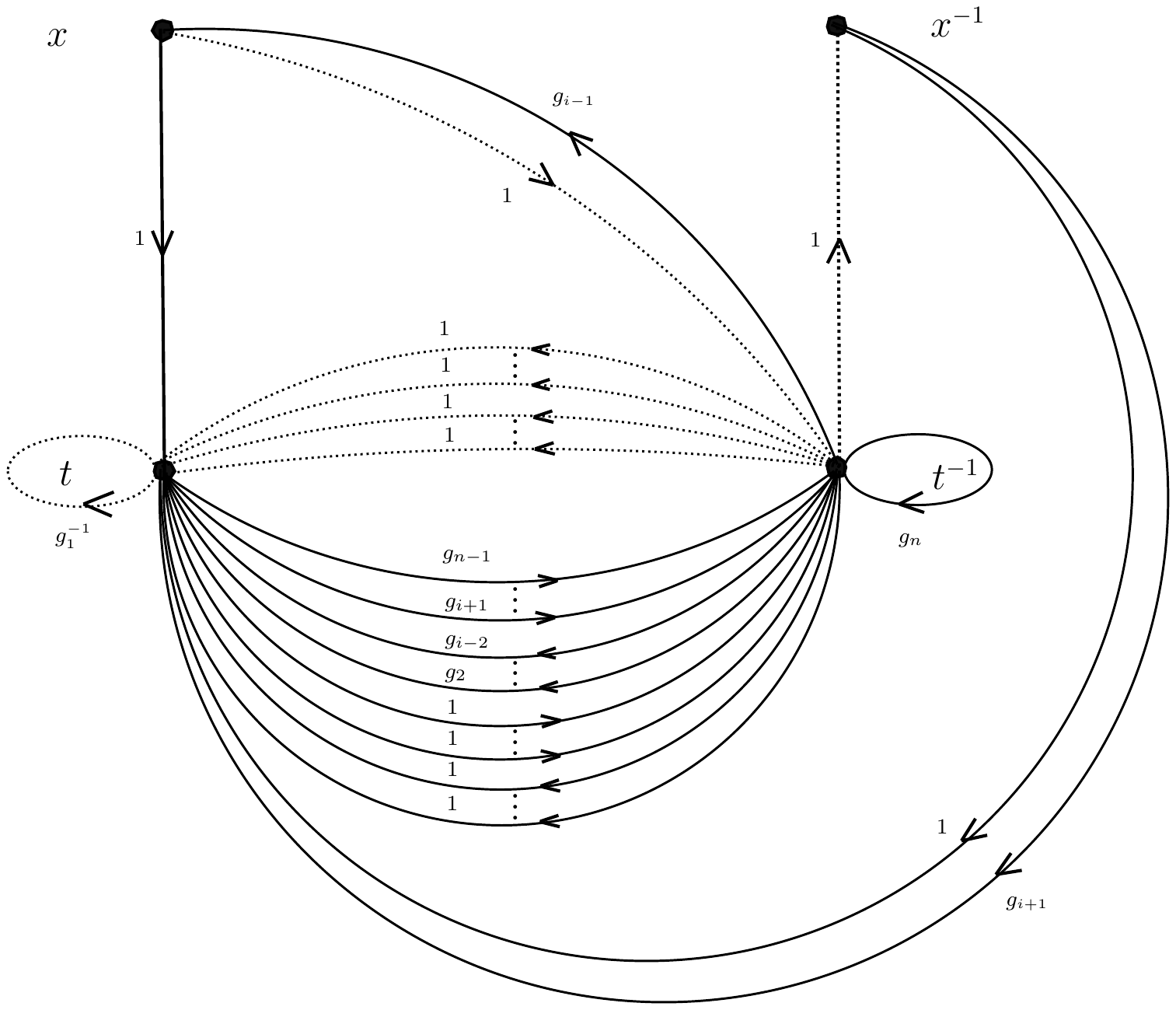}
	 		\vspace{-3.1in}
	 		\caption{Star graph $\Gamma$}
	 		\label{sgtype4_1_2}
	 		
	 	\end{center}
	 \end{figure} 
	 
%

	 We assign a weight function $\theta$ such that $\theta(g_1^{-1})=\theta(g_{n})=0$. Moreover the weight of the edges $x   \leftrightarrow   t$ and $t^{-1}   \leftrightarrow   x^{-1}$ with label $1$ is also zero. All other edges are assigned a weight $1$. It is clear from the star graph that all admissible cycles of weight less than two imply that the group is a torsion group. Hence $\mathcal{P}$ is aspherical.
	 
	 \item Let $m_i = m_n $ and $m_1 > m_{i-1}$. The star graph $\Gamma$ for $\mathcal{P}$ is given by Figure \ref{sgtype4_1_1_1}.
	 
	 \begin{figure}
	 	\begin{center}
	 		\vspace{-1.0in}
	 		\hspace*{-1.5in}
	 		\includegraphics[width=0.75\textwidth, height=0.75\textheight]{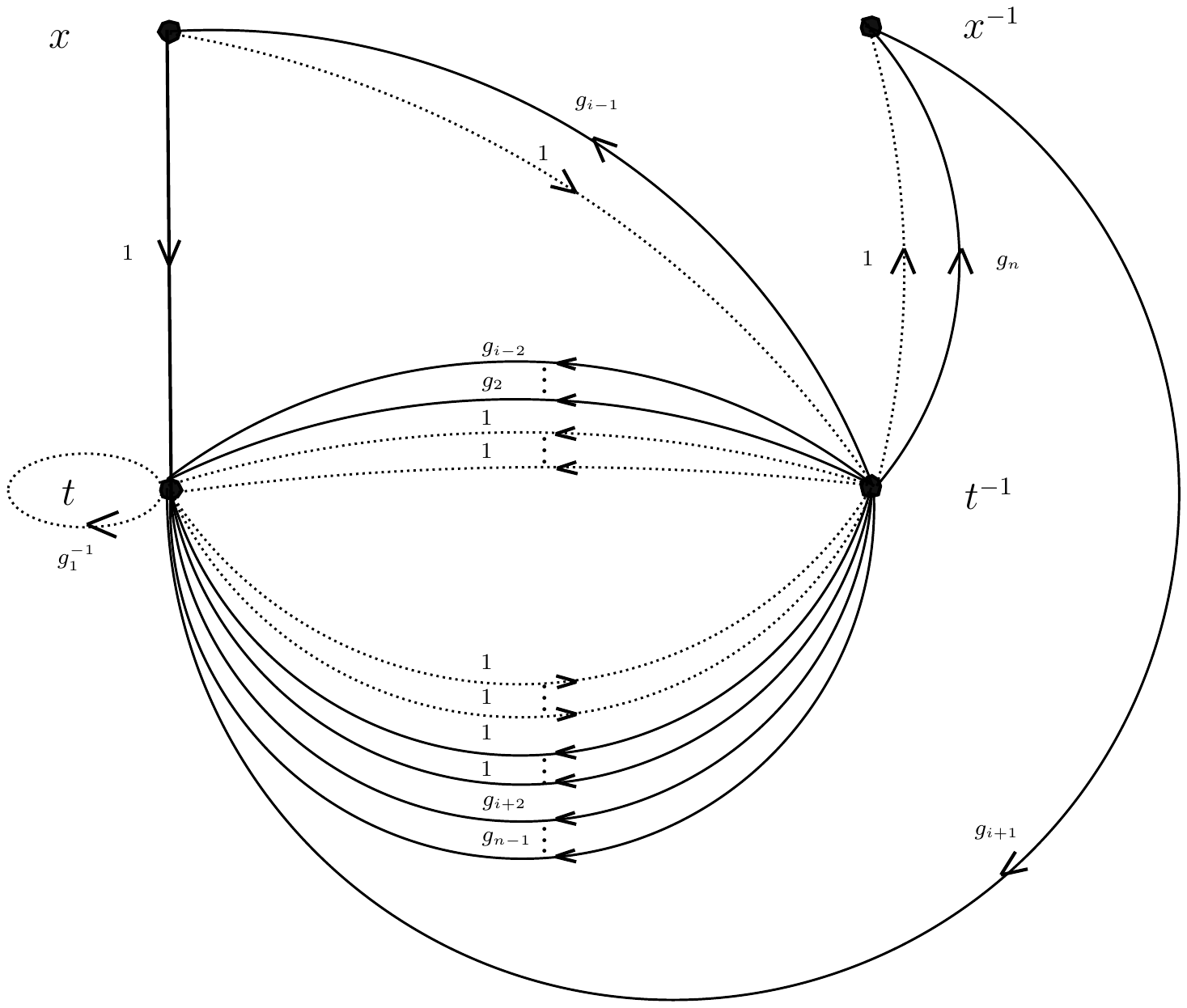}
	 		\vspace{-3.1in}
	 		\caption{Star graph $\Gamma$}
	 		\label{sgtype4_1_1_1}
	 		
	 	\end{center}
	 \end{figure}

%
	 
	 We assign a weight function $\theta$ such that $\theta(g_1^{-1})=\theta(g_{i+1})=\theta(g_{i-1})=0$. Moreover the weight of the edge $x   \leftrightarrow   t^{-1}$ with label $1$ is also zero. All other edges are assigned a weight $1$. It is clear from the star graph that $\mathcal{P}$ is aspherical for a torsion free group.
	 
	 Similarly we can prove the result for cases $m_i > m_n, \; m_1 = m_{i-1}$ and $m_i < m_n, \; m_1 = m_{i-1}$.
		 
	\end{enumerate}
	\end{proof}

\begin{lem}\label{lsgtype2} The equation $g_1tg_2t \cdots g_{i-1}t^{-1}g_i t g_{i+1} t \cdots g_{j-1}t^{-1}g_j t g_{j+1} t \cdots t g_n t^{-1}=1$ is solvable if $g_i=g_j =g_1^{-1}$.
\end{lem}
\begin{proof}

	Let $$\mathcal{P}=\langle A, t~|~ g_1tg_2t \cdots g_{i-1}t^{-1}g_i t g_{i+1} t \cdots g_{j-1}t^{-1}g_j t g_{j+1} t \cdots t g_n t^{-1}  \rangle$$ be the relative presentation corresponding to the given equation. We will use the weight test to show that the presentation $\mathcal{P}$ is aspherical. Substitute $x=t^{-1} g_1^{-1} t$ to get $$\mathcal{P}=\langle A, t~|~ x^{-1}g_2t \cdots g_{i-1} x g_{i+1} t \cdots g_{j-1} x g_{j+1} t \cdots t g_n =1 =t^{-1}g_{1}^{-1}tx^{-1}  \rangle.$$  The star graph $\Gamma$ for $\mathcal{P}$ is given by Figure \ref{addsgtype2}.

	\begin{figure}
		\begin{center}
			\vspace{-1.0in}
			\hspace*{-1.5in}
			\includegraphics[width=0.75\textwidth, height=0.75\textheight]{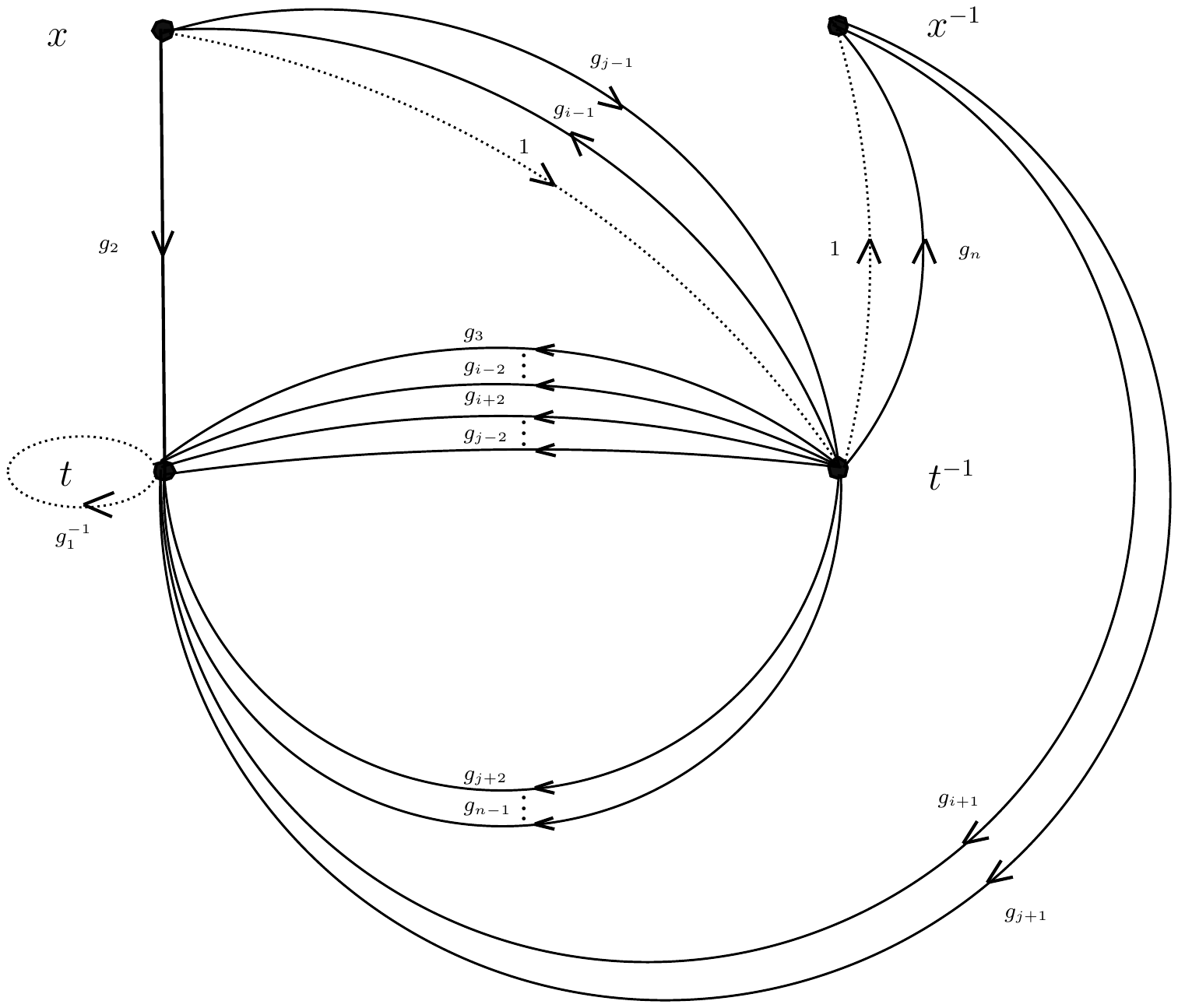}
			\vspace{-3.1in}
			\caption{Star graph $\Gamma$}
			\label{addsgtype2}
			
		\end{center}
	\end{figure}


	 We assign a weight function $\theta$ such that $\theta(g_1^{-1})=\theta(g_{2})=\theta(g_{n})=\theta(1)=0$, where $1$ is the label for the edge $t^{-1}   \leftrightarrow   x^{-1}$. All other edges are assigned a weight $1$. It is clear from the star graph that all admissible cycles of weight less than two imply that the group is a torsion group. Hence $\mathcal{P}$ is aspherical over a torsion free group.
	
\end{proof}

\begin{lem}\label{lsgtype5} Let $m_1, \cdots , m_n $ be positive integers. Then the equation $$g_1t^{m_1}g_2t^{m_2} \cdots g_{i-1}t^{-m_{i-1}}g_i t^{m_i} g_{i+1} \cdots g_{j-1}t^{-m_{j-1}}g_j t^{m_j} g_{j+1} \cdots g_n t^{-m_n}=1$$ is solvable if $g_i = g_j =g_1^{-1}$ and any one of the following holds.
	\begin{enumerate}
		\item $m_i > m_n $, $m_1 > m_{i-1}$, $m_{i-1} < m_{j-1} $ and $m_{j} > m_{i}$
		\item $m_i = m_n $ and $m_1 > m_{i-1}$,  $m_{i-1} > m_{j-1} $ and $m_{j} > m_{i}$
	\end{enumerate}
\end{lem}
\begin{proof}
	
	Let $$\mathcal{P}=\langle A, t~|~ g_1t^{m_1}g_2t^{m_2} \cdots g_{i-1}t^{-m_{i-1}}g_i t^{m_i} g_{i+1} \cdots g_{j-1}t^{-m_{j-1}}g_j t^{m_j} g_{j+1} \cdots g_n t^{-m_n}  \rangle$$ be the relative presentation corresponding to the given equation. We will show that $\mathcal{P}$ is aspherical by using the weight test. Substitute $x=t^{-m_{i-1}} g_1^{-1} t^{m_i}$ to get $\mathcal{P}=\langle A, t~|~ t^{m_i - m_n} x^{-1} t^{m_1 - m_{i-1}}g_2t^{m_2} \; \; \cdots t^{m_{i-2}}g_{i-1} x g_{i+1} t^{m_{i+1}} \cdots \\   t^{m_{j-2}}g_{j-1} t^{m_{i-1}-m_{j-1}} x t^{m_{j}-m_{i}}  g_{j+1} \cdots t^{m_{n-1}} g_n=1  =t^{-m_{i-1}}g_{1}^{-1}t^{m_{i}}x^{-1}  \rangle.$ The proof is given in separate cases.
	
	\begin{enumerate}
		\item Let $m_i > m_n $, $m_1 > m_{i-1}$, $m_{i-1} < m_{j-1} $ and $m_{j} > m_{i}$. The star graph $\Gamma$ for $\mathcal{P}$ is given by Figure \ref{sgtype4_2}.
		
			\begin{figure}
			\begin{center}
				\vspace{-1.0in}
				\hspace*{-1.5in}
				\includegraphics[width=0.75\textwidth, height=0.75\textheight]{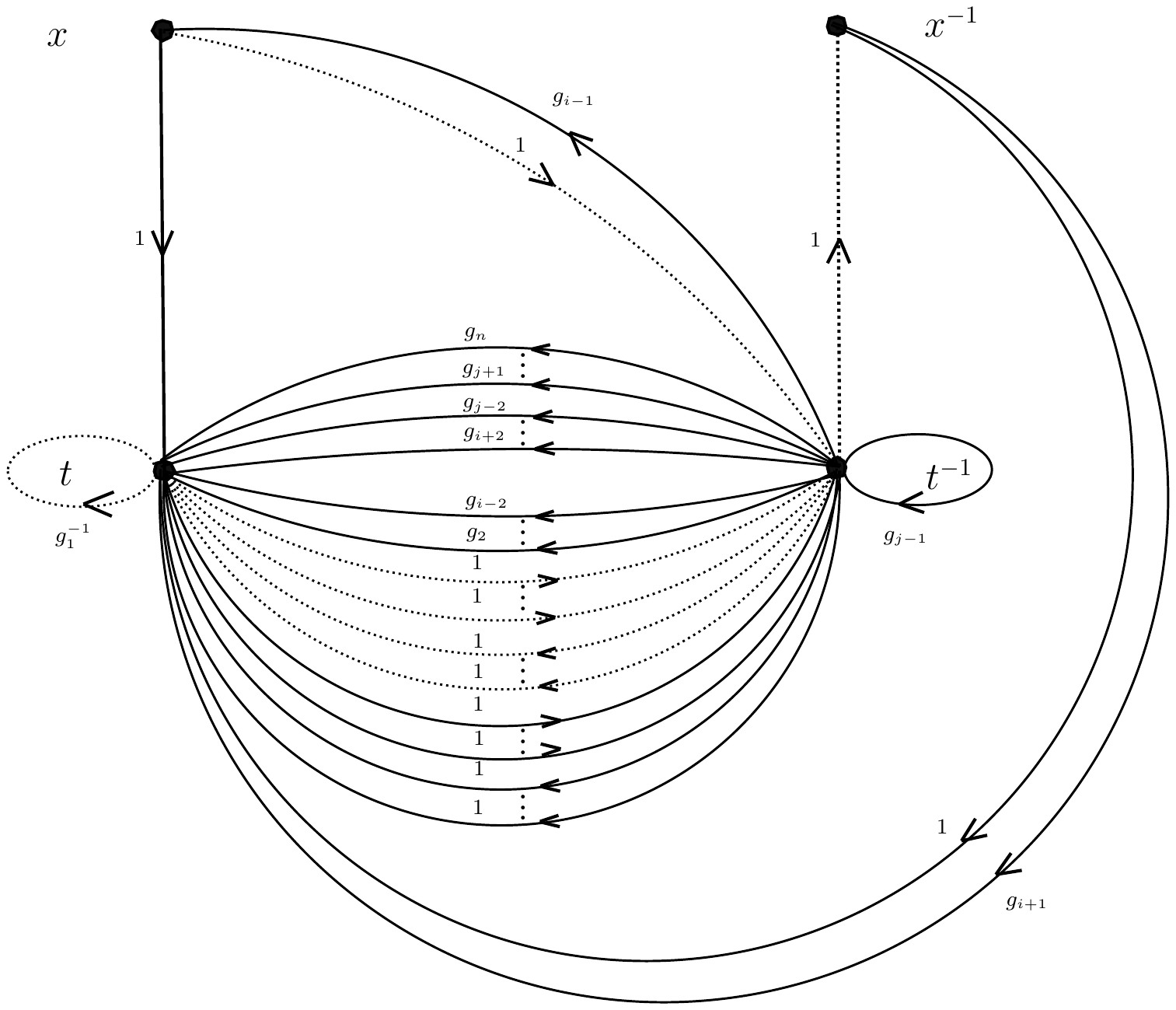}
				\vspace{-3.1in}
				\caption{Star graph $\Gamma$}
				\label{sgtype4_2}
				
			\end{center}
		\end{figure}

		
			We assign a weight function $\theta$ such that $\theta(g_1^{-1})=\theta(g_{j-1})=0$. Moreover the weight of the edges $t   \leftrightarrow   x$  and $t^{-1}  \leftrightarrow   x^{-1}$ with label $1$ is also zero. All other edges are assigned a weight $1$. It is clear from the star graph that all admissible cycles of weight less than two imply that the group is a torsion group. Hence $\mathcal{P}$ is aspherical over a torsion free group.
		
		\item Let $m_i = m_n $ and $m_1 > m_{i-1}$,  $m_{i-1} > m_{j-1} $ and $m_{j} > m_{i}$. The star graph $\Gamma$ for $\mathcal{P}$ is given by Figure \ref{sgtype4}.

			\begin{figure}
			\begin{center}
				\vspace{-1.0in}
				\hspace*{-1.5in}
				\includegraphics[width=0.75\textwidth, height=0.75\textheight]{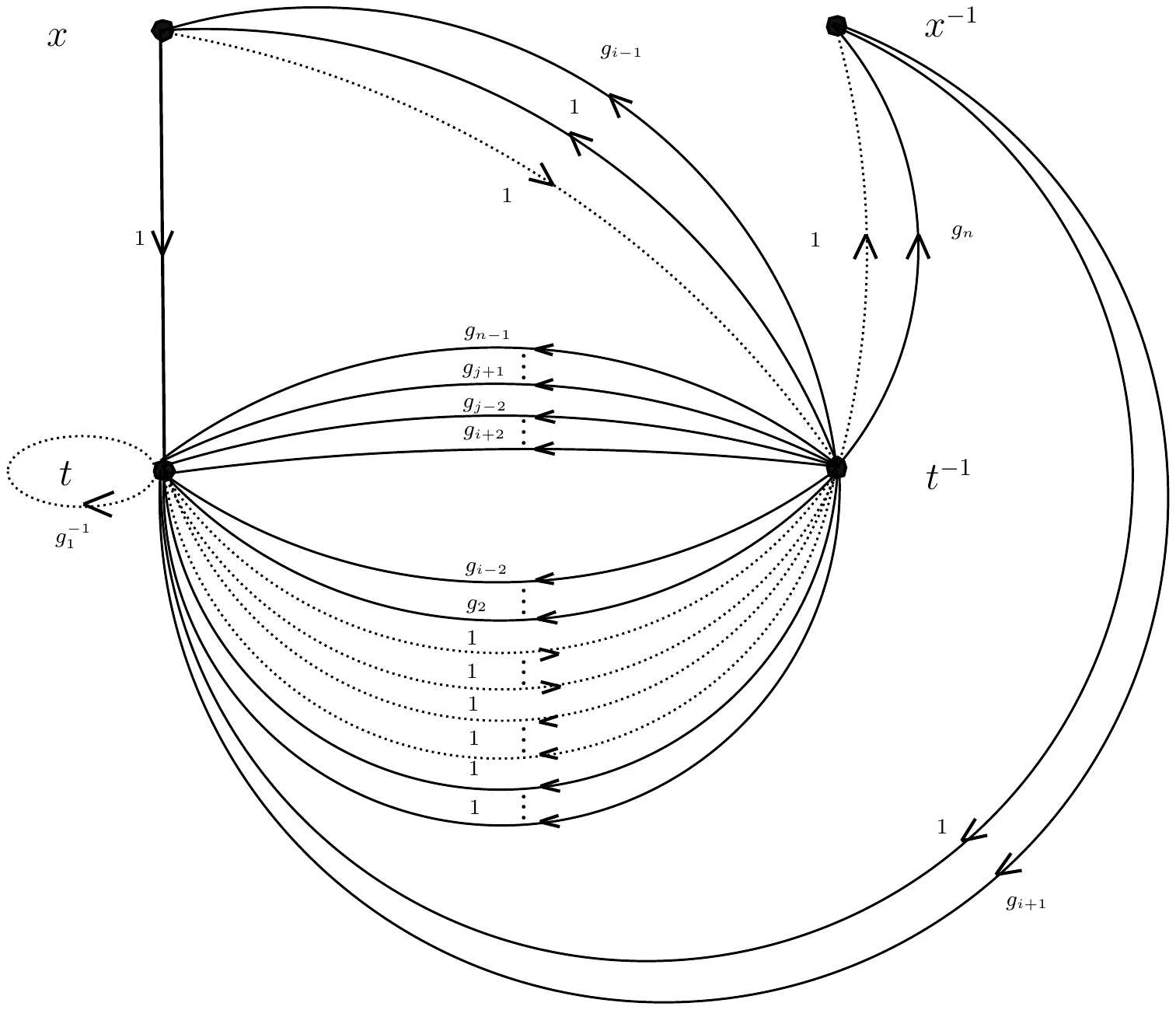}
				\vspace{-3.1in}
				\caption{Star graph $\Gamma$}
				\label{sgtype4}
				
			\end{center}
		\end{figure} 
		
%
		We assign a weight function $\theta$ such that $\theta(g_1^{-1})=\theta(g_{n})=0$. Moreover the weight of the edges $x   \rightarrow   t$ and $t^{-1}   \rightarrow   x^{-1}$ with label $1$ is also zero. All other edges are assigned a weight $1$. It is clear that the star graph $\mathcal{P}$ is aspherical over a torsion free group.

	\end{enumerate}
	
\end{proof}

\begin{rem}
	The results of Lemma \ref{lsgtype4} and Lemma \ref{lsgtype5} are valid even if we replace $t^{m_k}$ with $\Pi_{i=1}^{l} \; a_{i}t^{m_i}$. The proof is similar to the one given above. 
\end{rem}

\begin{theo}
	Let $\displaystyle s(t)=g_{1}E_{k_1}g_{2}E_{k_2} \, \cdots \, g_{n}E_{k_n}$ such that $g_1 , g_2,  \, \cdots  \, ,g_n \in G$ and  \\ $E_{k_i}=t a_{k_{i-1}+1} t \, \cdots \, t a_{k_i}t^{-1}$ with $k_0 =0$ and $k_i \geq 1$ for all $i \geq 1$. If $g_{j}=g_{1}^{-1}$ for all $j \geq 2$ then $s(t)=1$ has a solution over $G$.
\end{theo}
\begin{proof}
	
	First assume that $k_i \geq 2$ for all $i \geq 1$. The relative presentation for the above equation is given by $$\mathcal{P}=\langle A, t~|~ g_{1}E_{k_1}g_{2}E_{k_2} \, \cdots \, g_{n}E_{k_n}  \rangle.$$  We will show that the presentation $\mathcal{P}$ is aspherical. Expand this equation by replacing the values of $E_{k_{i}}$ to get $$g_{1}ta_{1}ta_{2}t \; \cdots \; ta_{k_{1}}t^{-1}g_{2}ta_{k_{1}+1}t \; \cdots \; ta_{k_{2}}t^{-1}g_{3}t \; \cdots \; ta_{k_{n}}t^{-1}=1. $$ Use the condition $g_{j}=g_{1}^{-1}$ for all $j \geq 2$ to get  $$g_{1}ta_{1}ta_{2}t \; \cdots \; ta_{k_{1}}t^{-1}g_{1}^{-1}ta_{k_{1}+1}t \; \cdots \; ta_{k_{2}}t^{-1}g_{1}^{-1}t \; \cdots \; ta_{k_{n}}t^{-1}=1. $$
	Substitute $x=t^{-1} g_1^{-1} t$ to get $$\mathcal{P}=\langle A, t~|~ x^{-1}a_1t \; \cdots \; ta_{k_{1}} x a_{k_{1}+1} t \; \cdots \; ta_{k_{2}} x a_{k_{2}+1} t \; \cdots \; ta_{k_{n}} =1 =t^{-1}g_{1}^{-1}tx^{-1}  \rangle.$$  The star graph $\Gamma$ for $\mathcal{P}$ is given by Figure \ref{sgtype1_9}.

	\begin{figure}
		\begin{center}
			\vspace{-1.0in}
			\hspace*{-1.5in}
			\includegraphics[width=0.75\textwidth, height=0.75\textheight]{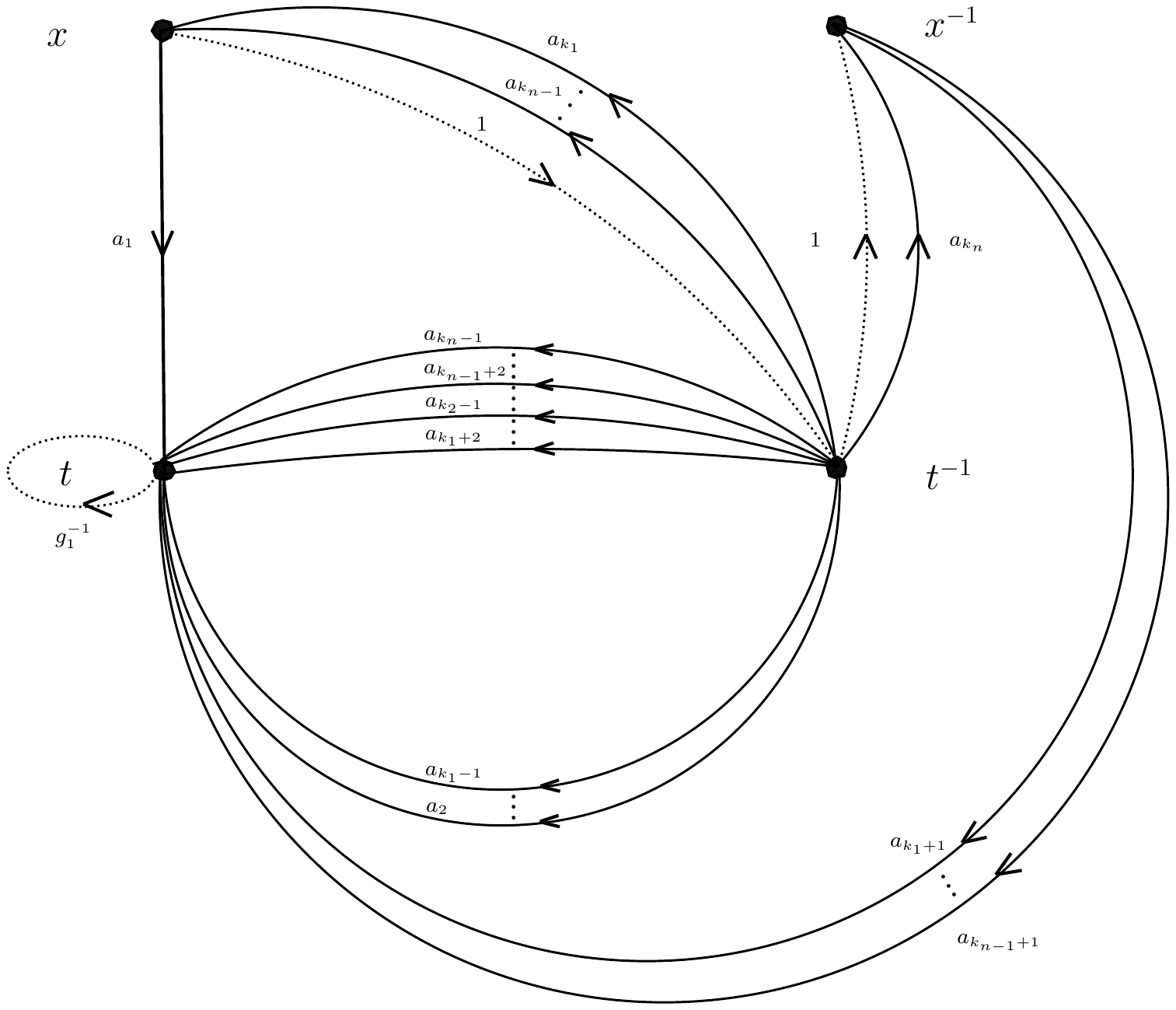}
			\vspace{-3.1in}
			\caption{Star graph $\Gamma$}
			\label{sgtype1_9}
			
		\end{center}
	\end{figure}

	
	We assign a weight function $\theta$ such that $\theta(g_1^{-1})=\theta(a_{1})=\theta(a_{k_{n}})=0$. Moreover the weight of the edge $t^{-1}   \rightarrow   x^{-1}$ with a label $1$ is also zero. All other edges are assigned a weight $1$. Now it is clear from the star graph that all admissible cycles of weight less than two imply that the group is a torsion group. Hence $\mathcal{P}$ is aspherical over a torsion free group.
	
	The proves for $k_i = 1$ for all $i \geq 1$ and for $k_i = 1$ for some $i \geq 1$ are similar therefore we omit the details.

\end{proof}

\begin{theo}
	Suppose $m_{1}, m_{2}, \; \cdots \;  ,m_{k_{n}}$ are positive integers such that $m_{k_{1}} > m_{k_i}$ and $m_{k_{i}+1} > m_{k_{1}+1}$ for all $i \geq 2$ with $m_{k_{1}+1}=m_{k_{n}}$. Let $\displaystyle s(t)=g_{1}E_{k_1}g_{2}E_{k_2} \, \cdots \, g_{n}E_{k_n}$ such that $g_1 , g_2  \, \cdots  \, g_n \in G$ and $E_{k_i}=t^{m_{k_{i-1}+1}} a_{k_{i-1}+1} t^{m_{k_{i-1}+2}} \, \cdots \, t^{m_{k_{i}-1}} a_{k_i}t^{-m_{k_{i}}}$ with $k_0 =0$ and $k_i \geq 2$ for all $i \geq 1$. If $g_{j}=g_{1}^{-1}$ for all $j \geq 2$ then $s(t)=1$ has a solution over $G$.
\end{theo}

\begin{proof}
	
	The relative presentation for the above equation is given by $$\mathcal{P}=\langle A, t~|~ g_{1}E_{k_1}g_{2}E_{k_2} \, \cdots \, g_{n}E_{k_n}  \rangle.$$  We will show that the presentation $\mathcal{P}$ is aspherical. Expand this equation by replacing the values of $E_{k_{i}}$ to get 
	
	\begin{align*}
	g_{1}t^{m_1}a_{1}t^{m_2} \; \cdots \; t^{m_{k_{1}-1}}a_{k_{1}}t^{-m_{k_1}}g_{2}t^{m_{k_{1}+1}} a_{k_{1}+1}t^{m_{k_{1}+2}}  \; \cdots \; t^{m_{k_{2}-1}}a_{k_{2}}t^{- m_{k_{2}}}&g_{3}t^{m_{k_{2}+1}} \; \cdots \; \\& t{m_{k_{n}-1}}a_{k_{n}}t^{-m_{k_{n}}}=1.
	\end{align*}
	
	 Use the condition $g_{j}=g_{1}^{-1}$ for all $j \geq 2$ to get 
	 
	 \begin{align*}
	  g_{1}t^{m_1}a_{1}t^{m_2} \; \cdots \; t^{m_{k_{1}-1}}a_{k_{1}}t^{-m_{k_1}}g_1^{-1}t^{m_{k_{1}+1}}a_{k_{1}+1}t^{m_{k_{1}+2}} \; \cdots \; t^{m_{k_{2}-1}} & a_{k_{2}}t^{- m_{k_{2}}}g_1^{-1}t^{m_{k_{2}+1}} \; \cdots \; \\& t{m_{k_{n}-1}}a_{k_{n}}t^{-m_{k_{n}}}=1.
	  \end{align*}
	  
	Substitute $x=t^{-m_{k_1}}g_1^{-1}t^{m_{k_{1}+1}}$ to get
	
	\begin{align*}
	 t^{m_{k_{1}+1}-m_{k_{n}}}x^{-1} t^{m_{1}-m_{k_{1}}} a_{1}t^{m_2} \; \cdots \; t^{m_{k_{1}-1}}a_{k_{1}} x a_{k_{1}+1}&t^{m_{k_{1}+2}} \; \cdots \; t^{m_{k_{2}-1}} a_{k_{2}} t^{m_{k_{1}}-m_{k_{2}}}  x \\& t^{m_{k_{2}+1}-m_{k_{1}+1}} \; \cdots \;  t{m_{k_{n}-1}}a_{k_{n}}=1. 
	 \end{align*}
	
	Using $m_{k_{1}+1}=m_{k_{n}}$ we obtain 
	\begin{align*}
	x^{-1} t^{m_{1}-m_{k_{1}}} a_{1}t^{m_2} \; \cdots \; t^{m_{k_{1}-1}}a_{k_{1}} x a_{k_{1}+1}t^{m_{k_{1}+2}} \; \cdots \; t^{m_{k_{2}-1}}a_{k_{2}} t^{m_{k_{1}}-m_{k_{2}}} & x t^{m_{k_{2}+1}-m_{k_{1}+1}} \; \cdots \; \\& t{m_{k_{n}-1}}a_{k_{n}}=1.
	\end{align*}
  All other conditions of $m_{i}$'s ensure that powers of $t$ are positive unless specified. The star graph $\Gamma$ is given by Figure \ref{sgtype1_10}. 
  
 \begin{figure}
  \begin{center}
  	\vspace{-1.0in}
  	\hspace*{-1.5in}
  	\includegraphics[width=0.75\textwidth, height=0.75\textheight]{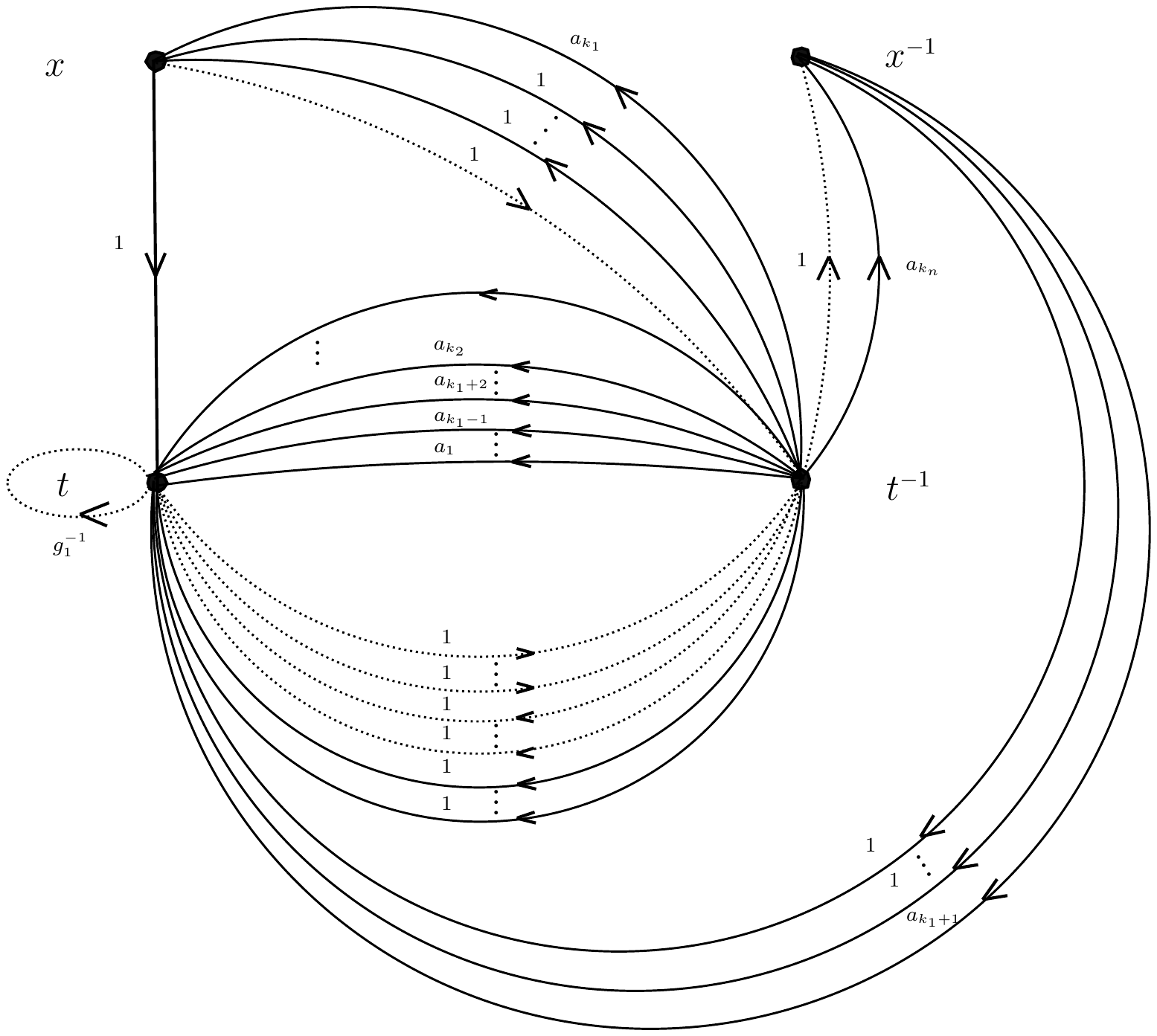}
  	\vspace{-2.7in}
  	\caption{Star graph $\Gamma$}
  	\label{sgtype1_10}
  	
  \end{center}
\end{figure}

	
	We assign a weight function $\theta$ such that $\theta(g_1^{-1})=\theta(a_{k_{n}})=0$. Moreover the weight of edges $t^{-1}   \rightarrow   x^{-1}$ and $x \rightarrow t$ with a label $1$ is also zero. All other edges are assigned a weight $1$. Now it is clear from the star graph that all admissible cycles of weight less than two imply that the group is a torsion group. Hence $\mathcal{P}$ is aspherical over a torsion free group.

\end{proof}


\begin{thebibliography}{99}

\bibitem{S} J. R. Stallings,  A graph theoretic lemma and group embedding, Annals of Mathematics studies, 111 (1987), 145-155.

\bibitem{BE} M. Bibi, M. Edjvet, Equation of length seven over torsion-free groups, J. Group Theory 21 (2018), 147-164.

\bibitem{ABIA}
M. Bibi, M. F. Anwar, S. Iqbal, M. S. Akram, Solution of a non-singular equation of length $8$ over torsion free groups, Preprint.

\bibitem{BP} W.A.Bogley, S.J.Pride, Aspherical relative presentations, Proceedings of the Edinburgh Mathematical Society, 35 (1992), 1-39. 

\bibitem{E}
A. Evangelidou, The solution of length five equations over groups, Comm.
in Alg. 35 (2007), 1914-1948.

\bibitem{K}
S. K. Kim, On the asphericity of length-6 relative presentations with torsion-free coefficients. Proceedings of the Edinburgh Mathematical Society 51.1 (2008), 201-214.

\bibitem{P}
M. I. Prishchepov, On small length equations over torsion-free groups, Internat. J. Algebra Comput. 4 (1994), no. 4, 575-589.

\bibitem{IK}
S.V. Ivanov and A. A. Klyachko, Solving equations of length at most six over
torsion-free groups, J. Group Theory 3 (2000), no. 3, 329-337.

\bibitem{BH}
S. D. Brodski˘ı and J. Howie, One-relator products of torsion-free groups, Glasg.
Math. J. 35 (1993), no. 1, 99-104.

\bibitem{levin}
F. Levin, Solutions of equations over groups, Bull. Amer. Math. Soc. 68 (1962), 603-604.

\end{thebibliography}
\end{document}